\documentclass[8pt]{article}

\usepackage[english]{babel}

\usepackage[letterpaper,top=2cm,bottom=2cm,left=2cm,right=2cm,marginparwidth=1.75cm]{geometry}

\usepackage{amsmath}
\usepackage{mathtools}
\mathtoolsset{showonlyrefs=false}
\usepackage{amssymb}
\usepackage{graphicx}
\usepackage{tikz}
\usepackage{pgfplots}
\pgfplotsset{compat=1.16}
\usepackage{mathrsfs}
\usepackage{ulem}
\mathtoolsset{showonlyrefs=true}

\usepackage{algorithm, algpseudocode, algorithmicx}

\usepackage[algo2e]{algorithm2e}

\usepackage[colorlinks=true, allcolors=blue]{hyperref}
\usepackage{color}
\usepackage{here}
\usepackage{epsfig}
\usepackage{ulem}
\usepackage[justification=centering]{caption}

\usepackage{amsthm}
\theoremstyle{definition}
\newtheorem{dfn}{Definition}[section]
\newtheorem{problem}{Problem}
\newtheorem{notation}{Notation}[section]

\newtheorem{lemma}[notation]{Lemma}
\newtheorem{theorem}[notation]{Theorem}

\newtheorem{remark}[dfn]{Remark}
\newtheorem{conj}[dfn]{Conjecture}

\newtheorem*{theorem4.1}{Theorem 4.1 of \cite{endo2024guaranteed}}

\usepackage{color}

\title{The second Dirichlet eigenvalue is simple on every non-equilateral triangle }

\author{Ryoki Endo\thanks{Graduate School of Science and Technology, Niigata University, Niigata, Japan
(endo@m.sc.niigata-u.ac.jp).}, Xuefeng Liu\thanks{School of Arts and Sciences, Tokyo Woman's Christian University, Tokyo, Japan (xfliu@cis.twcu.ac.jp) (Corresponding author).}}

\begin{document}
\date{}
\maketitle

\begin{abstract}
The Dirichlet eigenvalues of the Laplacian on a triangle that collapses into a line segment diverge to infinity. In this paper, to track the behavior of the eigenvalues during the collapsing process of a triangle, we establish a quantitative estimate for the Dirichlet eigenvalues on collapsing triangles. As an application, we solve the open problem concerning the simplicity of the second Dirichlet eigenvalue for nearly degenerate triangles, offering a complete solution to Conjecture 6.47 posed by R. Laugesen and B. Siudeja in A. Henrot's book ``Shape Optimization and Spectral Theory".
\end{abstract}

% \tableofcontents

% \medskip

% (1 Hour)

% Pen: $\checkmark\checkmark\Box\Box\Box$ $\Box\Box\Box\Box\Box$ $\Box\Box\Box\Box\Box$ $\Box\Box\Box\Box\Box$ $\Box\Box\Box\Box\Box$ $\Box\Box\Box\Box\Box$

% PC: $\checkmark\checkmark\checkmark\checkmark\checkmark$ $\Box\Box\Box\Box\Box$ $\Box\Box\Box\Box\Box$ $\Box\Box\Box\Box\Box$ $\Box\Box\Box\Box\Box$ $\Box\Box\Box\Box\Box$

\section{Introduction}
The rich relationship between Laplacian eigenvalues and shapes gave birth to the field of spectral geometry, which continues to attract researchers from various disciplines.
In this paper, we provide a computer-assisted proof for a conjecture about Dirichlet eigenvalues posed by R. Laugesen and B. Siudeja in Henrot's book ``Shape Optimization and Spectral Theory" \cite{henrot2017shape}:

\begin{conj}[Conjecture 6.47 of \cite{henrot2017shape}]
\label{len:main-conjecture}
The second Dirichlet eigenvalue is simple on every non-equilateral triangle.
\end{conj}
In \cite{endo2024guaranteed}, we provided a partial result confirming the conjecture, except for the case of collapsing triangles:
\begin{theorem4.1}
The second Dirichlet eigenvalue is simple for every non-equilateral triangle with its minimum normalized height \footnote{The minimum normalized height of a triangle is the height measured relative to its longest side, with the triangle scaled such that the longest side has unit length.} greater than or equal to $\tan(\pi/60)/2$.
\end{theorem4.1}
This paper completes the proof by covering the case of nearly degenerate triangles:
\begin{theorem}\label{thm:degenerate-case-for-the-main-conjecture}
The second Dirichlet eigenvalue is simple for every non-equilateral triangle with its minimum normalized height less than or equal to 
% $\tan(\pi/60)/2$
$\tan(\pi/60)/2$
.
\end{theorem}

To achieve this, we investigated the behavior of eigenvalues over collapsing triangles.
As a domain collapses, Dirichlet eigenvalues diverge to infinity, and their asymptotic behavior depends on the geometry of the domain and the perturbations driving its collapse.

Several studies have explored the behavior of eigenvalues of collapsing domains. Jerison \cite{jerison1991first} investigated the nodal line of eigenfunctions of the Dirichlet Laplacian, demonstrating that the nodal line of the second Dirichlet eigenfunction touches the boundary in a collapsing convex domain. In \cite{jerison1995diameter}, it is shown that the nodal line is located near the zero of an associated ordinary differential equation, with estimates for the first and second eigenvalues derived in terms of the eigenvalues of an ordinary differential operator.  Friedlander and Solomyak \cite{friedlander2009spectrum} analyzed the asymptotic behavior of Dirichlet eigenvalues for strips with a specific width profile as they collapse.

\medskip
\medskip

Below, we reference the results obtained by Ourmières-Bonafos \cite{ourmieres2015dirichlet}.

For $s\in(-1,1)$ and $t>0$, let $T(s,t)$ be the triangular domain with vertices $(-1,0),(1,0)$ and $(s,t)$; see Figure \ref{fig:shape-of-triangle-Tst}.
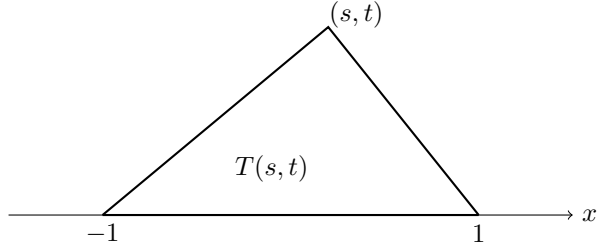
\begin{figure}[H]
    \centering
\begin{tikzpicture}[scale=2.5]
    % Draw the base of the triangle
    \draw[thick] (-1, 0) -- (1, 0);
    
    % Shift the vertex slightly to the right
    \coordinate (shifted_vertex) at (0.2, 1);
    
    % Draw the sides of the triangle with the shifted vertex
    \draw[thick] (-1, 0) -- (shifted_vertex) -- (1, 0);
    
    % Label the shifted vertex
    \node at (0.35, 1.07) {$(s,t)$};
    
    % Label the base points
    \node at (-1, -0.1) {$-1$};
    \node at (1, -0.1) {$1$};
    
    % Label h(x) at some point x
    \node at (-0.1, 0.25) {$T(s,t)$};
    
    % Add labels for the axes
    \draw[->] (-1.5, 0) -- (1.5, 0) node[right] {$x$};
\end{tikzpicture}
    \caption{Shape of triangle $T(s,t)$}
    \label{fig:shape-of-triangle-Tst}
\end{figure}

Ourmi\`eres-Bonafos derived the following asymptotic expansions for the eigenvalues on collapsing triangles:
\begin{lemma}{[\cite{ourmieres2015dirichlet}, Theorem 1.2, Proposition 2.4]}
\label{thm:bonafos}
The $k$-th Dirichlet eigenvalue over $T(s,t)$, denoted by $\lambda_k(s,t)$, admits the following expansion:
\begin{equation}\label{eq:bonafos}
\lambda_k(s,t)\sim t^{-2} \left( \pi^2 + (2\pi^2)^{2/3} \kappa_k(s) t^{2/3} + \beta_{3,k} t + \cdots \right)~~(t\to 0+).
\end{equation}
Here, $\kappa_k(s)~~(k=1,2,\cdots)$ are the $k$-th eigenvalues of the Schr\"{o}dinger operator $l_s^{\text{mod}}$ defined in $H^2(\mathbb{R})$. Moreover, each eigenvalue of $l_s^{\text{mod}}$ is simple, and the functions $s\mapsto \kappa_k(s)$ are analytic on $(-1,1)$ for all $k=1,2,\cdots$. For the definition of $l_s^\text{mod}$, see \eqref{eq:def-of-l-s-mod} in \ref{appendix-airy}.
\end{lemma}
The symbol \(\sim\) in \eqref{eq:bonafos} represents the convergence defined below.
\begin{notation}
We write
\begin{equation}
\lambda_k(s, t) \sim \sum_{j \geq 0} \Gamma_j(s) t^{j/3}~~(t \to 0+)
\end{equation}
if, for any $J \in \mathbb{N}$, there exist $C_J(s) > 0$ and $t_0(s) > 0$ such that for all $t \in (0, t_0(s))$, 
\begin{equation}
\left| \lambda_k(s, t) - \sum_{j=0}^J \Gamma_j(s) t^{j/3} \right| \leq C_J(s) t^{(J+1)/3}.
\end{equation}
Here, the $\Gamma_j$'s are the coefficients of the series.
\end{notation}

\medskip

While Lemma \ref{thm:bonafos} shows that $\lambda_2$ and $\lambda_3$ are separated in the limit case of a completely collapsed triangle, it does not establish their simplicity for small $t$ with a given value\footnote{Let $\mu_2 := \lim_{t\to 0+}t^{4/3}(\lambda_2(s,t)-\pi^2/t^2)$ and $\mu_3 := \lim_{t\to 0+}t^{4/3}(\lambda_3(s,t)-\pi^2/t^2)$. Since Lemma \ref{thm:bonafos} shows that the second and third eigenvalues of the Schr\"{o}dinger operator $l_s^{\text{mod}}$ are simple, we have $\mu_2 < \mu_3$. Thus, there exists $t_0 > 0$ such that
$$t^{\frac{4}{3}}\left(\lambda_2(s,t)-\frac{\pi^2}{t^2}\right) < \frac{\mu_2+\mu_3}{2} < t^{\frac{4}{3}}\left(\lambda_3(s,t)-\frac{\pi^2}{t^2}\right)\text{  for all $ t \in (0,t_0]$.}$$ 
 However, this argument does not provide an explicit value of $t_0$.}.
To overcome this limitation, we derive the following explicit  estimates for the $k$-th Dirichlet eigenvalues on the collapsing triangle:
letting $t_0=\tan(\pi/60)/2$, 
\begin{equation}\label{eq:estimation-for-mu-t-intro}
\frac{\bar{\mu}_k(s)}{1 + t_0^\frac{2}{3}/(3\pi^2) \bar{\mu}_k(s)}
\leq t^{\frac{4}{3}}\left(\lambda_k(s,t) - \frac{\pi^2}{t^2}\right) \leq \hat\mu_k^{t_0}(s) \quad \text{for } t \in (0, t_0]~~~(k=1,2,\cdots),
\end{equation}
where $\bar\mu_k(s)$ is the $k$-th eigenvalue of another Schr\"{o}dinger operator on $\mathbb{R}$, and $\hat\mu_k^{t_0}(s)$ is the $k$-th eigenvalue of a Schr\"{o}dinger operator over a bounded interval; see the definitions in \eqref{eq:problem2} and \eqref{eq:problem4}. 
% The eigenvalue problems for $\bar\mu_k(s)$ and $\hat\mu_k^{t_0}(s)$ are defined in the next section. 
It is worth pointing out that the values or bounds of the involved eigenvalues are all computable by utilizing the recently developed methods for rigorous eigenvalue estimation \cite{liu2024guaranteed}.

The estimation for eigenvalues 
in \eqref{eq:estimation-for-mu-t-intro} allows us to separate $\lambda_2(s,t)$ and $\lambda_3(s,t)$ for $t \in (0, t_0]$, confirming the simplicity of the second eigenvalue for nearly degenerate triangles. 
% The concrete value of $t_0$ is given in Section \ref{sec:computer-assister-proof}.

\medskip

The remainder of this paper is organized as follows. Section 2 introduces three eigenvalue problems to be used in the derivation of the inequality \eqref{eq:estimation-for-mu-t-intro}. Section 3 establishes inequality \eqref{eq:estimation-for-mu-t-intro}. 
Section 4 provides a computer-assisted proof for Conjecture \ref{len:main-conjecture}. Section 5 concludes the paper by summarizing our results.  The code for the computer-assisted proof is available at \url{https://ganjin.online/ryoki/DirichletSimplicityCollapse}.

\section{Preliminary}\label{sec:preliminary}

We introduce the standard notation for Sobolev spaces to begin our discussion. Let $T \subset \mathbb{R}^2$ be a triangular domain. The space $L^2(T)$ denotes all real-valued square-integrable functions on $T$, and $H^1(T)$ represents the space of all functions in $L^2(T)$ whose weak derivatives are also in $L^2(T)$. Additionally, $H^1_0(T)$ is the subspace of $H^1(T)$ consisting of functions that vanish on the boundary of $T$. The $L^2$-norm of $v \in L^2(T)$ is denoted by $\|v\|_{T}$, and the inner product in $L^2(T)$ or $(L^2(T))^2$ is written as $(\cdot,\cdot)_T$. The gradient operator for functions in $H^1(T)$ is denoted by $\nabla$. Note that $(\nabla\cdot,\nabla\cdot)_T$ defines an inner product on $H^1_0(T)$ due to the imposed boundary conditions.

The weak formulation of the Dirichlet eigenvalue problem of Laplacian is to find $u \in H^1_0(T) \setminus \{0\}$ and $\lambda > 0$ such that 
\begin{equation}
\label{eq:eigenvalue-problem}
 (\nabla u, \nabla v)_T = \lambda (u, v)_T \quad \forall v \in H^1_0(T).
\end{equation}
 Since the inverse of the Laplacian is a compact and self-adjoint operator, the spectral theorem guarantees that the problem \eqref{eq:eigenvalue-problem} has a countably infinite set of eigenvalues, which can be arranged as $0 < \lambda_1(T) < \lambda_2(T) \leq \lambda_3(T) \leq \cdots$. 

\begin{figure}[H]
    \centering
\begin{tikzpicture}[scale=2.5]
    % Draw the base of the triangle
    \draw[thick] (-1, 0) -- (1, 0);
    
    % Shift the vertex slightly to the right
    \coordinate (shifted_vertex) at (0.2, 1);
    
    % Draw the sides of the triangle with the shifted vertex
    \draw[thick] (-1, 0) -- (shifted_vertex) -- (1, 0);
    
    % Label the shifted vertex
    \node at (0.35, 1.07) {$(s,t)$};
    
    % Label the base points
    \node at (-1, -0.2) {$-1$};
    \node at (1, -0.2) {$1$};
    
    % Label h(x) at some point x
    \draw[dashed] (-0.5, 0) -- (-0.5, 0.42);
    \node at (-0.3, 0.25) {$h(x)$};
    
    % Label x
    \node at (-0.5, -0.2) {$x$};
    
    % Add labels for the axes
    \draw[->] (-1.5, 0) -- (1.5, 0) node[right] {$x$};
\end{tikzpicture}
    \caption{Shape of triangle $T^t$}
    \label{fig:shape-of-triangle}
\end{figure}
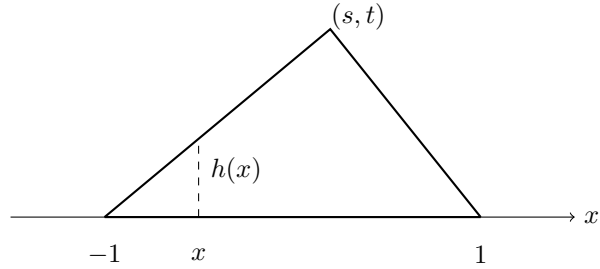

For simplicity of notation, denote by $T^t$ the triangular domain $T(s,t)$ for a fixed $s\in (-1,1)$.
The function \( h(x) \) on the interval \( I=[-1,1] \) is given by  
\begin{equation}
h(x) = 
\begin{cases}
\frac{t}{1+s}(x + 1) & \text{for } x \in [-1, s], \\
-\frac{t}{1-s}(x - 1) & \text{for } x \in [s, 1],
\end{cases}
\end{equation}  
which defines the shape of the triangle $T^t$; see Figure \ref{fig:shape-of-triangle}. The weak derivative of $h(x)$ is denoted by $h'(x)$.
Also, define the scaled interval $I^t = \left[{-t^{-2/3} (1+s)},{t^{-2/3} (1-s)}\right]$.

\medskip

Let us consider the asymptotic behavior of the eigenvalue $\lambda^t_k$ as the height $t$ approaches zero, treating the interval $I$ as the base of $T^t$. 
To obtain the inequality \eqref{eq:estimation-for-mu-t-intro}, we consider 3 eigenvalue problems for the shifted and scaled Laplacian. The relation among these eigenvalues will be discussed in Section 3 and summarized in Table \ref{tab:relations-eigenvalues}.

\begin{problem}
Find \( u \in H^1_0(T^t) \) and \( \mu > 0 \) such that
\begin{equation} \label{eq:problem1}
t^{4/3} \left[ (\nabla u, \nabla v)_{T^t} - \dfrac{\pi^2}{t^2} (u, v)_{T^t} \right] = \mu (u, v)_{T^t} \quad \text{for all } v \in H^1_0(T^t).
\end{equation}
\end{problem}

For \( t > 0 \), let \( \mu^t_k \) denote the \( k \)-th eigenvalue of Problem 1. Then, we have 
 \begin{equation}
\label{eq:relation_mu_and_lambda}
    \mu_k^t=t^{4/3}\left(\lambda_k^t-\frac{\pi^2}{t^2}\right),~ \text{or }~ 
     \lambda_k^t  = t^{-4/3}\left(\mu_k^t + \frac{\pi^2}{t^2}\right) .
\end{equation}
From the asymptotic expansion \eqref{eq:bonafos} and 
the above relation between $\mu_k^t$ and $\lambda_k^t$, the asymptotic behavior 
of eigenvalue $\mu_k^t$ is 
$$
\mu_k^t \to  (2\pi^2)^{2/3}\kappa_k\text{ as } 
 t \to 0+ .
$$

\medskip

The essential behavior of eigenvalues of Problem 1 will be studied through the one of Problem 2, which is defined on a one-dimensional interval.

\begin{problem}
Find \( \hat{u} \in H^1_0(I^t) \) and \( \hat{\mu} > 0 \) such that
\begin{align} \label{eq:problem2}
\int_{I^t}
\hat u'\hat v'~dx +  \int_{I^t}V(t,x)\hat{u}\hat{v}    ~dx 
 = \hat{\mu} \int_{I^t}\hat{u} \hat{v} ~ dx\quad \text{for all } \hat{v} \in H^1_0(I^t),
\end{align}
where
\begin{equation}\label{def-V-t-x}
    V(t,x):=t^{4/3}\left( \dfrac{ \pi^2 }{h(t^{2/3} x+s)^2 } + \dfrac{ (3 + 4\pi^2) h'(t^{2/3} x+s)^2 }{ 12  h(t^{2/3} x+s)^2 } - \dfrac{ \pi^2 }{ t^2 }\right )~~(x\in I^t).
\end{equation}
\end{problem}
Let \( \hat{\mu}^t_k \) denote the \( k \)-th eigenvalue of Problem 2.
This problem was proposed in \cite{friedlander2009spectrum} to analyze the asymptotic behavior of eigenvalues over narrow strips.
Note that an upper bound for $\hat{\mu}^t_k$ can be easily obtained using the Rayleigh--Ritz method.
  
% The eigenvalue problem 2 on one-dimensional interval $I^t$ is essentially the one defined in $W(T^t)$ on a two-dimensional triangular domain. 

Lemma \ref{lem:problem1-2} tells that 
Problem 2 is obtained by restricting the function space \(H^1_0(T^t)\) in Problem 1 to the subspace \(W(T^t)\) defined by  
% The following problem is obtained by restricting the function space \(H^1_0(T^t)\) in Problem 1 to the subspace \(W(T^t)\) defined by  
\[
W(T^t) := \left\{ v(x) \sqrt{\frac{2}{h(x)}} \sin\left(\frac{\pi y}{h(x)}\right) \,\middle|\, v \in H^1_0(I) \right\}~(\subset H^1_0(T^t)).
\]
% \begin{lemma}\label{lem:problem1-2}
%     The pair \((\hat{\mu}, \hat{u})\) is an eigen-pair of Problem 2 if and only if \(\left(\hat{\mu}, \hat{u}(t^{-2/3}(x-s))\sqrt{\frac{2}{h(x)}} \sin\left(\frac{\pi y}{h(x)}\right)\right)\) is an eigen-pair of the following eigenvalue problem:  
%     Find \(u \in W(T^t)\) and \(\hat{\mu} > 0\) such that  
%     \begin{equation} \label{eq:problem1-in-lemma}
%     t^{4/3} \left[ (\nabla u, \nabla v)_{T^t} - \frac{\pi^2}{t^2} (u, v)_{T^t} \right] = \hat{\mu} (u, v)_{T^t} \quad \text{for all } v \in W(T^t).
%     \end{equation}
% \end{lemma}

\begin{lemma}\label{lem:problem1-2} The pair \((\hat{\mu}, \hat{u})\) is an eigen-pair of Problem 2 if and only if \[ \left(\hat{\mu},\, \hat{u}\Bigl(t^{-2/3}(x-s)\Bigr)\sqrt{\frac{2}{h(x)}} \sin\left(\frac{\pi y}{h(x)}\right)\right) \] is an eigen-pair of the following eigenvalue problem:  Find $  u \in W(T^t)$ and  $\hat{\mu} > 0$ such that 
\[t^{4/3} \left[ (\nabla u, \nabla v)_{T^t} - \frac{\pi^2}{t^2} (u, v)_{T^t} \right] = \hat{\mu} (u, v)_{T^t} \quad \text{for all } v \in W(T^t).\]
\end{lemma}

\begin{proof}
    The proof is provided in Appendix Lemma \ref{lem:problem1-2}.
\end{proof}

In the limit as \( t \to 0+ \) in Problem 2, we obtain the following problem:
\begin{problem}
Find \( \bar{u} \in H^1(\mathbb{R}) \) and \( \bar{\mu} > 0 \) such that
\begin{equation} \label{eq:problem4}
\int_{-\infty}^{\infty} \bar{u}' \bar{v}' \, dx +  \int_{-\infty}^{\infty} V(x) \bar{u} \bar{v} \, dx = \bar{\mu} \int_{-\infty}^{\infty} \bar{u} \bar{v} \, dx \quad \text{for all } \bar{v} \in H^1(\mathbb{R}),
\end{equation}    
where
\begin{equation}\label{eq:def-of-V}
V(x):=2\pi^2\left(\frac{1}{1-s}\mathbf{1}_{\mathbb{R}_+}+\frac{1}{1+s}\mathbf{1}_{\mathbb{R}_-}\right)|x|.
\end{equation}
Here, $\mathbf{1}_{A}$ denotes the indicator function of a set $A(\subset \mathbb{R})$. The \( k \)-th eigenvalue of Problem 3 is denoted by \( \bar{\mu}_k \).
\end{problem}

Upper and lower bounds of $\bar{\mu}_k$ with concrete values will be obtained using known facts for the Schr\"{o}dinger operator with the Airy function; see Lemma \ref{thm:airy-eigen}.

In the next section, we establish the following inequality for the eigenvalues \( \mu^t_k \), \( \hat{\mu}^t_k \), and \( \bar{\mu}_k \):
\begin{equation}\label{eq:upper-lower-mut}
\frac{\bar{\mu}_k}{1 + t_0^\frac{2}{3} / (3\pi^2) \bar{\mu}_k}
\leq \mu_k^t \leq \hat\mu_k^{t_0} \quad \text{for } 0 < t \leq t_0.
\end{equation}
By applying this inequality, we separate \( \lambda^t_2 \) and \( \lambda^t_3 \) over nearly degenerate triangles, thereby resolving Conjecture \ref{len:main-conjecture}. 

\section{Upper and lower bounds for \( \mu^t_k \)}

First, we establish the upper bound \(\mu_k^t \leq \hat{\mu}_k^{t_0}\) in \eqref{eq:upper-lower-mut} by using the min-max principle and the monotonicity of \(\hat{\mu}_k^t\) of Problem 2  with respect to \(t\). Then, we obtain a lower bound estimate for $\mu^t_k$, by using the projection error estimate for the projection from \(H^1_0(T^t)\) onto \(W(T^t)\).

\subsection{Upper bound for \( \mu^t_k \)}

Note that the eigenvalues of Problem 2 can be characterized by the one defined on \( W(T^t) \) through Lemma \ref{lem:problem1-2}.  
From the min-max principle, $W(T^t)\subset H^1_0(T^t)$ implies $\mu_k^t \le \hat\mu_k^{t}$.

% \begin{align}
%     \mu_k^t 
%     &= \min_{V^{(k)}\subset H^1_0(T^t)}\max_{v\in V^{(k)}} \frac{t^{\frac{4}{3}}\left[\|\nabla v\|^2_{T^t}-\frac{\pi^2}{t^2}\|v\|^2_{T^t}\right]}{\|v\|^2_{T^t}}\\
%     &\leq
%     \min_{W^{(k)}\subset W(T^t)}\max_{v\in W^{(k)}} \frac{t^{\frac{4}{3}}\left[\|\nabla v\|^2_{T^t}-\frac{\pi^2}{t^2}\|v\|^2_{T^t}\right]}{\|v\|^2_{T^t}}
%     = \hat\mu_k^{t},
% \end{align}
% where $V^{(k)}$ and $W^{(k)}$ are the $k$-dimensional subspaces of $H^1_0(T^t)$ and $W(T^t)$, respectively.

\medskip

Denote by $R_t$ and $R$ the Rayleigh quotients for Problems 2 and 3, using $V(t,x)$ in \eqref{def-V-t-x} and $V(x)$ in \eqref{eq:def-of-V}, respectively. That is, 
\begin{equation}\label{eq:def-rt}
R_t[u]
\;:=\;\frac{\displaystyle \int_{I^{t}} \Bigl(\lvert u'(x)\rvert^2 
            + V(t,x)\,\lvert u(x)\rvert^2\Bigr)\,dx}
      {\displaystyle \int_{I^{t}} \lvert u(x)\rvert^2\,dx}~~(u\in H^1_0(I^t)), 
\end{equation}
and
\[
R[u]
\;:=\;\frac{\displaystyle \int_{\mathbb{R}} \Bigl(\lvert u'(x)\rvert^2 
            + V(x)\,\lvert u(x)\rvert^2\Bigr)\,dx}
      {\displaystyle \int_{\mathbb{R}} \lvert u(x)\rvert^2\,dx}~~(u\in H^1(\mathbb{R})).
\]

Let us first confirm the monotonicity and the convergence of $V(t,x)$ with respect to $t$.

\begin{lemma}\label{lem:monotonicity}
    For a fixed $x\in I^t$, we have the following properties about the potential $V(t,x)$
    \begin{description}
        \item[(i)]  $V(t,x)$ is monotonically increasing with respect to $t$.
        \item[(ii)]
        $\displaystyle \lim_{t\to 0+}V(t,x)=V(x)$.
    \end{description}
     
\end{lemma}

\begin{proof}
%     The simplified form of the expression for \( V(t,x) \) is:

% \begin{equation}
% V(t,x) =\begin{cases}
%     \frac{ -24 \pi^2 (s-1) x - 12 \pi^2 t^{2/3} x^2 + ( 3 + 4\pi^2 ) t^{4/3} }{ 12 \left( ( s - 1 )^2 + 2(s-1)t^{2/3} x + t^{4/3} x^2 \right) } & (x>0)\\
%     \frac{ -24 \pi^2 s x -12 \pi^2 t^{2/3} x^2 -24 \pi^2 x + ( 3 + 4\pi^2 ) t^{4/3} }{ 12 \left( ( s + 1 )^2 + 2 ( s + 1 ) t^{2/3} x + t^{4/3} x^2 \right) } & (x\leq 0)
% \end{cases}.
% \end{equation}

To show (i), let us compute the partial derivative of \( V(t, x) \) with respect to \( t \):

\begin{align}
    \frac{\partial}{\partial t}V(t,x) &= 
    \begin{cases}
    \displaystyle
        \frac{6x^2\pi^2t^{2/3}(3t^{-2/3}(1-s)-x) + (3 + 4 \pi^2)(1-s) t^{2/3}}{9 (t^{-2/3}(1-s)-x)^3t^{7/3}} & (x>0)\\
        \displaystyle
        \frac{6x^2\pi^2t^{2/3}(3t^{-2/3}(1+s)+x) + (3 + 4 \pi^2)(1 + s) t^{2/3}}{9 (t^{-2/3}(1+s)+x)^3t^{7/3}} & (x\leq 0)
    \end{cases}.
\end{align}
Since $x \in I^t = \left[{-t^{-2/3} (1+s)},{t^{-2/3} (1-s)}\right]$, it is easy to see
$\frac{\partial}{\partial t}V(t,x)\geq 0$. Thus,  $V(t,x)$ is monotonically increasing with respect to $t$.
The property (ii) follows from a simple computation.
\end{proof}

Using Lemma \ref{lem:monotonicity} regarding the potential $V(t,x)$, we obtain the following relations between the eigenvalues of Problems 2 and 3.
\begin{lemma}\label{thm:monotonicity-t}
For the $k$-th eigenvalues $\bar{\mu}_k$ and $\hat{\mu}_k^{t}$, we have
\begin{description}
    \item[(i)] $\bar{\mu}_k\leq\hat{\mu}_k^{t_1} \leq \hat{\mu}_k^{t_2}~~\text{for any $0<t_1\leq t_2$}.$
    \item[(ii)] $\lim_{t\to 0+}\hat{\mu}_k^t = \bar{\mu}_k.$
\end{description}

\end{lemma}

\begin{proof}
The proof of (i) is based on the monotonicity of the potential \( V(t, x) \) with respect to \( t \).

First, note that \( I^{t_2} \subset I^{t_1} \) for \( t_1 \leq t_2 \). Let \( \tilde{u} \in H^1_0(I^{t_1}) \) be the zero extension of \( u\in H^1_0(I^{t_2}) \) defined by
\[
\tilde{u}(x) =
\begin{cases}
u(x), & x \in I^{t_2}, \\
0, & x \in I^{t_1} \setminus I^{t_2}.
\end{cases}
\]
Then, from Lemma \ref{lem:monotonicity}, we have
\[
|\tilde u'(x)|^2 + V(t_1, x) |\tilde u(x)|^2 \leq |u'(x)|^2 + V(t_2, x) | u(x)|^2, \quad \forall u \in H^1_0(I^{t_2}).
\]
Therefore, for any \( u \in H^1_0(I^{t_2}) \), 
\[
R_{t_1}[\tilde u]\leq R_{t_2}[u].
\]
By the min-max principle, we deduce that
$\hat{\mu}_k^{t_1} 
\leq
\hat{\mu}_k^{t_2}$.

Similarly, the inequality $\bar{\mu}_k\leq \hat{\mu}_k^{t_1}$ can be shown by using the min-max principle and the property of the zero extension of \( u\in H^1_0(I^{t_1}) \):
\begin{equation}\label{eq:zero-extention}
\tilde{u}(x) =
\begin{cases}
u(x), & x \in I^{t_1} \\
0, & x \in \mathbb{R} \setminus I^{t_1}
\end{cases},
\quad R[\tilde{u}] \le R_{t_1}[u].
\end{equation}

Property (ii) can be shown using the continuity of $V(t,x)$ with respect to $t$.
% \textbf{Proof of (ii):}
% Let us observe that the eigenvalue \(\hat{\mu}_k^t\) converges to \(\bar{\mu}_k\).
% Take arbitrary sequence $t_n>0$ that converges to $0$.
% From the monotonicity obtained in (i), there exists the limit $ \mu^*:=\lim_{n\to\infty}\hat{\mu}_k^{t_n} $. Since the property (i) of this Lemma tells us $\mu^* \ge \bar{\mu}_k$, to show $\mu^* = \bar{\mu}_k$, it is sufficient to prove that $\mu^* \le \bar{\mu}_k$.

% To show this, 
% take any $u\in H^1_0(I^t)$ and let \( \tilde{u} \in H^1_0(\mathbb{R}) \) be the zero extension of \( u\in H^1_0(I^{t}) \).
% By property (ii) of Lemma~\ref{lem:monotonicity}, for any compact set \(K\subset \mathbb{R}\) there exists a sequence \(\{\varepsilon_{t_n}\}\) with \(\varepsilon_{t_n}\to 0\) as \(n\to\infty\) such that
% \[
% V(t_n,x) \leq V(x) + \varepsilon_{t_n} \quad \text{for all } x\in K.
% \]
%  Therefore, on any compact set of $\mathbb{R}$, we have
% \[
% |u'(x)|^2 + V(t_n,x)\,|u(x)|^2
% \;\le\;
% |\tilde u'(x)|^2 + \bigl(V(x) + \varepsilon_{t_n}\bigr)\,|\tilde{u}(x)|^2.
% \]
% By integrating over $I^{t_n}$, we obtain 
% \[
% \lim_{n\to \infty}R_{t_n}[u]
% \;\le\; R[\tilde{u}].
% \]
% Using the min-max principle, we deduce
% \[
% \mu^*
% \;=\;\lim_{n\to \infty}\hat{\mu}_k^{t_n}
% \;\le\; \bar{\mu}_k.
% \]
\end{proof}

\subsection{Lower bound for \( \mu^t_k \)}
In \cite{liu2015framework}, Liu introduced a method for evaluating lower bounds of Laplacian eigenvalues by employing a projection error estimate for a nonconforming finite element space. Here, we apply this method to the infinite-dimensional subspace $W(T^t)$ of $H^1_0(T^t)$ to derive a lower bound for  $\mu^t_k$ of Problem 1.

Let $\tilde a(\cdot,\cdot)$ and $\tilde b(\cdot,\cdot)$ be the bilinear forms on  $H^1_0(T^t)$ defined by
\begin{equation}
    \tilde a(u,v):=t^{\frac{4}{3}} \left[ (\nabla u, \nabla v)_{T^t} - \dfrac{\pi^2}{t^2} (u, v)_{T^t} \right],~~\tilde b(u,v):=(u,v)_{T^t}.
\end{equation}
Note that $\tilde a(\cdot,\cdot)$  forms an inner product on $H^1_0(T^t)$ due to the known fact $\lambda^t_k>\pi^2/t^2$; see \cite[Prop. 1.1]{ourmieres2015dirichlet}. 
Define the norms $\|\cdot\|_{\tilde a}$ and $\|\cdot\|_{\tilde b}$ over $H^1_0(T^t)$ by
\begin{equation}
    \|\cdot\|_{\tilde a}:=\sqrt{\tilde a(\cdot,\cdot)},~~
    \|\cdot\|_{\tilde b}:=\sqrt{\tilde b(\cdot,\cdot)}.
\end{equation}
Let \( \widetilde{P} \) be the projection operator from \( H^1_0(T^t) \) onto \( W(T^t) \) with respect to the inner product $\tilde a(\cdot,\cdot)$. 

First, we establish an estimate for the projection \( \widetilde{P} \).

\begin{lemma}\label{lem:projection-error-estimate}
For any $u\in H^1_0(T^t)$, the following estimate holds:
\begin{equation}
    \|u -\widetilde{P} u\|_{\tilde b} \leq \dfrac{ t^{ \frac{1}{3} } }{ \sqrt{3} \pi } \|u -\widetilde{P} u\|_{\tilde a}.
\end{equation}
\end{lemma}

\begin{proof}
For each fixed \( x \in [-1, 1] \), consider the orthonormal basis \( \{\phi_n(y; x)\}_{n\geq1} \) of \( L^2((0, h(x))) \) defined by
\[
\phi_n(y; x) = \sqrt{\dfrac{2}{h(x)}} \sin\left( \dfrac{n\pi y}{h(x)} \right).
\]
Then \( u(x, y) \) admits the Fourier series expansion
\[
u(x, y) = \sum_{n=1}^\infty \hat{u}_n(x) \phi_n(y; x),
\]
where
\[
\hat{u}_n(x) = \int_0^{h(x)} u(x, y) \phi_n(y; x) \, dy.
\]
Since \( W(T^t) \) consists of functions involving only \( \phi_1(y; x) \), the projection \( \widetilde{P} u \) retains only the \( n = 1 \) term:
\[
\widetilde{P} u(x, y) = \hat{u}_1(x) \phi_1(y; x).
\]
Therefore, the difference \( u - \widetilde{P} u \) is given by
\[
u - \widetilde{P} u = \sum_{n=2}^\infty \hat{u}_n(x) \phi_n(y; x).
\]
Using the orthonormality of \( \{\phi_n(y; x)\} \) in \( L^2(0, h(x)) \), we have
\begin{align*}
\|u - \widetilde{P} u\|_{\tilde{b}}^2 &= \int_{T^t} |u - \widetilde{P} u|^2 \, dx \, dy = \int_{-1}^1 \int_0^{h(x)} \left| \sum_{n=2}^\infty \hat{u}_n(x) \phi_n(y; x) \right|^2 dy \, dx \\
&= \int_{-1}^1 \sum_{n=2}^\infty |\hat{u}_n(x)|^2 dx.
\end{align*}
Next, we compute \( \tilde{a}(u - \widetilde{P} u, u - \widetilde{P} u) \). The bilinear form \( \tilde{a} \) satisfies
\[
\tilde{a}(u - \widetilde{P} u, u - \widetilde{P} u) = t^{4/3} \left( \|\nabla (u - \widetilde{P} u)\|_{L^2(T^t)}^2 - \dfrac{\pi^2}{t^2} \|u - \widetilde{P} u\|_{L^2(T^t)}^2 \right).
\]
Let us compute \( \|\nabla (u - \widetilde{P} u)\|_{L^2(T^t)}^2 \) by calculating the partial derivatives with respect to \( x \) and \( y \).

For the derivative with respect to \( y \), we have
\[
\dfrac{\partial \phi_n}{\partial y} = \dfrac{n\pi}{h(x)} \sqrt{\dfrac{2}{h(x)}} \cos\left( \dfrac{n\pi y}{h(x)} \right).
\]
Therefore, it follows that
\begin{equation}\label{eq:derivative-y}
\left\| \dfrac{\partial (u - \widetilde{P} u)}{\partial y} \right\|_{L^2(T^t)}^2 = \int_{-1}^1 \sum_{n=2}^\infty |\hat{u}_n(x)|^2 \left( \dfrac{n\pi}{h(x)} \right)^2 dx.
\end{equation}
For the derivative with respect to \( x \), we compute
\begin{equation}
\dfrac{\partial \phi_n}{\partial x} = -\dfrac{h'(x)}{2h(x)} \phi_n(y; x) - \dfrac{n\pi y h'(x)}{h(x)^2} \sqrt{\dfrac{2}{h(x)}} \cos\left( \dfrac{n\pi y}{h(x)} \right).
\end{equation}
Using orthogonality and standard integral identities, we find that
\[
\left\| \dfrac{\partial \phi_n}{\partial x} \right\|_{L^2(0, h(x))}^2 = \dfrac{h'(x)^2}{h(x)^2} \left( \dfrac{4 n^2 \pi^2 + 3}{6} \right).
\]
Therefore,
\begin{equation}\label{eq:derivative-x}
\left\| \dfrac{\partial (u - \widetilde{P} u)}{\partial x} \right\|_{L^2(T^t)}^2 = \int_{-1}^1 \sum_{n=2}^\infty \left( |\hat{u}_n'(x)|^2 + |\hat{u}_n(x)|^2 \left\| \dfrac{\partial \phi_n}{\partial x} \right\|_{L^2(0, h(x))}^2 \right) dx.
\end{equation}
Combining the contributions from \eqref{eq:derivative-y} and \eqref{eq:derivative-x}, we obtain
\begin{align*}
\|\nabla (u - \widetilde{P} u)\|_{L^2(T^t)}^2 &= \int_{-1}^1 \sum_{n=2}^\infty \left\{ |\hat{u}_n'(x)|^2 + |\hat{u}_n(x)|^2 \left[ \dfrac{h'(x)^2}{h(x)^2} \dfrac{4 n^2 \pi^2 + 3}{6} + \left( \dfrac{n\pi}{h(x)} \right)^2 \right] \right\} dx.
\end{align*}
Thus, we can write
\[
\|u - \widetilde{P} u\|_{\tilde{a}}^2 = t^{4/3} \int_{-1}^1 \sum_{n=2}^\infty \left\{ |\hat{u}_n'(x)|^2 + |\hat{u}_n(x)|^2 \left[ \dfrac{h'(x)^2}{h(x)^2} \dfrac{4 n^2 \pi^2 + 3}{6} + \left( \dfrac{n\pi}{h(x)} \right)^2 - \dfrac{\pi^2}{t^2} \right] \right\} dx.
\]
Using the fact that \( h(x) \leq t \) and \( h'(x) \) is bounded, we can estimate the coefficient of \( |\hat{u}_n(x)|^2 \) for \( n \geq 2 \) as
\[
\left( \dfrac{n\pi}{h(x)} \right)^2 - \dfrac{\pi^2}{t^2} \geq \dfrac{4\pi^2}{t^2} - \dfrac{\pi^2}{t^2} = \dfrac{3\pi^2}{t^2}.
\]
Therefore,
\[
\|u - \widetilde{P} u\|_{\tilde{a}}^2 \geq t^{4/3} \dfrac{3\pi^2}{t^2} \int_{-1}^1 \sum_{n=2}^\infty |\hat{u}_n(x)|^2 dx = \dfrac{3\pi^2}{t^{2/3}} \|u - \widetilde{P} u\|_{\tilde{b}}^2.
\]
Rearranging this inequality yields
\[
\|u - \widetilde{P} u\|_{\tilde{b}}^2 \leq \dfrac{t^{2/3}}{3\pi^2} \|u - \widetilde{P} u\|_{\tilde{a}}^2,
\]
and taking square roots gives
\[
\|u - \widetilde{P} u\|_{\tilde{b}} \leq \dfrac{t^{1/3}}{\sqrt{3}\pi} \|u - \widetilde{P} u\|_{\tilde{a}}.
\]
This completes the proof.
\end{proof}
Using the newly obtained projection error estimate in Lemma \ref{lem:projection-error-estimate} and following the proof of \cite[Thm. 2.1.]{liu2015framework} almost verbatim, we obtain the following lower bound:
\begin{lemma}\label{thm:ch-estimation}
For each \( k \geq 1 \), the eigenvalues \( \mu_k^t \) of Problem 1 and \( \hat{\mu}_k^t \) of Problem 2 satisfy
\begin{equation}\label{eq:eigenvalue_estimate}
\mu_k^t \geq \dfrac{ \hat{\mu}_k^t }{ 1 + t^\frac{2}{3}/(3\pi^2) \hat{\mu}_k^t }.
\end{equation}
\end{lemma}

\begin{proof}

From Lemma \ref{lem:projection-error-estimate}, we have
\begin{equation}\label{eq:projection-error-estimate}
    \| u - \widetilde{P} u \|_{\tilde b} \leq C \| u - \widetilde{P} u \|_{\tilde a}~~\text{for all \( u \in H^1_0(T^t) \)},
\end{equation}
where $C:=\dfrac{t^{1/3}}{\sqrt{3}\pi}$.
Let \(  \hat{u}_k  \in W(T^t)\) be an $L^2$-orthonormal eigenfunction corresponding to \( \hat{\mu}_k^t\).

Define \( E^h_{k-1} = \mathrm{span}\{ \hat{u}_1, \dots, \hat{u}_{k-1} \} \). Let \((E^h_{k-1})^\bot \) be the orthogonal complement of $E^h_{k-1}$ in $H^1_0(T^t)$ with respect to $\tilde a(\cdot,\cdot)$. Noting that
\begin{equation}
    \hat{\mu}_k^t=\inf_{u\in (E^h_{k-1})^\bot}\frac{\| u \|^2_{\tilde{a}}}{\| u \|^2_{\tilde{b}}},
\end{equation}
we have
\begin{equation}\label{eq:lower_bound_M_new}
\| \widetilde{P} u \|_{\tilde{a}} \geq \sqrt{\hat{\mu}_k^t} \| \widetilde{P} u \|_{\tilde{b}}~~\text{for any \( u \in  (E^h_{k-1})^\bot \)}.
\end{equation}
Decompose \( u \) as
\[
u = \widetilde{P} u + (u - \widetilde{P} u),
\]
with \( \widetilde{P} u \in W(T^t) \) and \( u - \widetilde{P} u \in (W(T^t))^\bot \), the \( \tilde{a} \)-orthogonal complement of \( W(T^t) \) in \( H^1_0(T^t) \). Then,
\begin{align}
\| u \|_{\tilde{b}} &\leq \| \widetilde{P} u \|_{\tilde{b}} + \| u - \widetilde{P} u \|_{\tilde{b}}. \label{eq:N_decomp_new}
\end{align}
Using \eqref{eq:lower_bound_M_new} and the projection error estimate \eqref{eq:projection-error-estimate}, we have
\begin{equation}\label{eq:error_Norm_new}
\| u \|_{\tilde{b}} \leq \frac{1}{\sqrt{\hat\mu_k^t}}\| \widetilde{P} u \|_{\tilde{a}} +  C \| u - \widetilde{P} u \|_{\tilde{a}},
\end{equation}
which leads to
\begin{align}
\| u \|_{\tilde{b}}^2 &\leq \left(\frac{1}{\hat\mu_k^t}+C^2\right)(\| \widetilde{P} u \|_{\tilde{a}}^2 +  \| u - \widetilde{P} u \|_{\tilde{a}}^2)\\
&=
\left(\frac{1}{\hat\mu_k^t}+C^2\right)\| u  \|_{\tilde{a}}^2.
\end{align}
Therefore, we have
\begin{equation}
    \frac{\|u\|_{\tilde{a}}^2}{\|u\|_{\tilde{b}}^2}\geq \frac{\hat\mu_k^t}{1+C^2\hat\mu_k^t}~~~\mbox{for any }u\in (E_{k-1}^h)^\bot.
\end{equation}
Thus, from the max-min principle, it follows that
\begin{equation}
    \mu_k^t\geq\frac{\hat\mu_k^t}{1+C^2\hat\mu_k^t}.
\end{equation}
\end{proof}

The following theorem summarizes the results obtained above:
\begin{theorem}\label{thm:eigenvalue_estimate-cor}
The eigenvalue \( \mu_k^t \) of Problem 1 satisfies
\begin{equation}\label{eq:main-estimation}
\frac{\bar{\mu}_k}{1+ t_0^\frac{2}{3}/(3\pi^2) \bar{\mu}_k}
\leq \mu_k^t \leq \hat\mu_k^{t_0} ~\mbox{ for any }t\in (0,t_0].
\end{equation}
\end{theorem}
\begin{proof}
    This result follows directly from Lemmas \ref{thm:monotonicity-t} and \ref{thm:ch-estimation}.
\end{proof}

\begin{remark}
    From the property (i) of Lemma~\ref{thm:monotonicity-t} and Theorem~\ref{thm:ch-estimation}, we have 
    \begin{equation}
    \frac{\bar{\mu}_k}{1+ t^\frac{2}{3}/(3\pi^2) \bar{\mu}_k}
    \leq \mu_k^t \leq \hat\mu_k^{t} \quad (t>0).
    \end{equation}
    Since property (ii) of Lemma~\ref{thm:monotonicity-t} implies that $\lim_{t\to 0+}\hat\mu_k^{t}=\bar{\mu}_k$, sending $t\to 0+$ in the above inequality yields
    \begin{equation}
        \lim_{t\to 0+}\mu_k^{t}=\bar{\mu}_k.
    \end{equation}
    Therefore, we have the following asymptotic expansion of $\lambda^t_k$:
    \begin{equation}\label{eq:bonafos-mu-bar}
    \lambda^t_k\sim t^{-2} \left( \pi^2 + \bar{\mu}_k t^{2/3} \right)~~(t\to 0+).
    \end{equation}
    Comparing \eqref{eq:bonafos-mu-bar} with the asymptotic expansion in Lemma~\ref{thm:bonafos}, we obtain the following relation:
    \begin{equation}\label{eq:kappa-mu-bar-relation}
        (2\pi^2)^{2/3}\kappa_k(s) =  \bar{\mu}_k(s).
    \end{equation}
\end{remark}

\medskip

In the next section, using a method to obtain explicit bounds of $\bar{\mu}_k$ and $\hat{\mu}_k^t$, we will provide a computer-assisted proof for the degenerate case of Conjecture \ref{len:main-conjecture}.

\section{Computer-Assisted Proof for the Degenerate Case of Conjecture \ref{len:main-conjecture}}\label{sec:computer-assister-proof}

In this section, we present a computer-assisted proof of Theorem \ref{thm:degenerate-case-for-the-main-conjecture}.

Consider a triangular domain \( T(s,t) \) with vertices at \((-1,0)\), \((1,0)\), and \((s,t)\). Since the simplicity of eigenvalues is invariant under isometry and scaling, without loss of generality, we assume that the vertex \((s,t)\) lies within the set defined by
\begin{align}
    \Omega := \left\{ (s,t) \in \mathbb{R}^2 \;\middle|\; (s+1)^2 + t^2 \leq 4,\; t > 0,\; s \geq 0 \right\}.
\end{align}
We define the following subsets of \( \Omega \):
\begin{align}
    \Omega_{\text{down}} &:= \left\{ (s,t) \in \Omega \;\middle|\; 0 \leq s < 1,\; 0 < t \leq \tan\left(\frac{\pi}{60}\right) \right\}.
\end{align}
Figure~\ref{fig:moduli-space-of-triangle} illustrates the moduli spaces \( \Omega \) and \( \Omega_{\text{down}} \).
Note that every triangle determined by $(s,t)$ in $\Omega_{\text{down}}$ has its minimum normalized height $\le \tan\left({\pi}/{60}\right) /2$.
To prove Theorem \ref{thm:degenerate-case-for-the-main-conjecture}, it suffices to prove that \( \lambda_2(s,t) < \lambda_3(s,t) \) for all \( (s,t) \in \Omega_{\text{down}} \).

\begin{figure}[H]
    \centering
\begin{tikzpicture}[scale=2.5]
  %-------------------------
  % Compute the x-coordinate of the intersection (point D) 
  % between the circle (center (-1,0), radius 2) and the horizontal line y=0.2.
  % The circle equation is (x+1)^2 + 0.2^2 = 4, so:
  %      x = sqrt(4 - 0.04) - 1.
  \pgfmathsetmacro{\xD}{sqrt(4-0.04)-1}  % x-coordinate of the intersection (point D)
  
  %-------------------------
  % Fill the region bounded by:
  %  1. The x-axis from the origin (0,0) to B=(1,0),
  %  2. The portion of the circular arc (with center (-1,0) and radius 2)
  %     from B=(1,0) up to the point where y=0.2 (i.e. D=(\xD,0.2)),
  %  3. The horizontal segment from D back to (0,0.2),
  %  4. And then the vertical segment from (0,0.2) down to the origin.
  %
  % Note: On the circle, y = 2 sin θ so that y = 0.2 when θ = arcsin(0.1) ≈ 5.739°.
  \fill[gray!20]
    (0,0) -- (1,0)                % from the origin to B along the x-axis
    arc (0:5.739:2)               % draw the arc from B up to point D; this arc uses
                                 % center (-1,0) since (1,0) = (-1,0)+(2,0)
    -- (0,0.2) -- cycle;          % then draw a line from D (which is (xD,0.2)) to (0,0.2)
                                 % and close the cycle back to the origin
  
  %-------------------------
  % Draw the boundaries for clarity:

  % Draw the segment from A = (-1,0) to (0,0)
  \draw (-1,0) -- (0,0);
  
  % Draw the x-axis segment from (0,0) to B=(1,0) (thick line)
  \draw[thick] (0,0) -- (1,0);
  
  % Draw the full circular arc (from B=(1,0) to (0,√3))
  \draw (-1,0) ++(2,0) arc (0:60:2);
  
  % Draw the vertical line from (0,√3) down to the origin (thick line)
  \draw[thick] (0,{sqrt(3)}) -- (0,0);
  
  % Draw the dotted line from A=(-1,0) to (0,{sqrt(3)}) for reference
  \draw[thick, dotted] (-1,0) -- (0,{sqrt(3)});
  
  % Draw the horizontal segment from (0,0.2) to the arc (point D)
  \draw[dotted] (0,0.2) -- (\xD,0.2);
  
  %-------------------------
  % Label key points:
  \node[below left] at (-1,0) {$A(-1,0)$};
  \node[below] at (1,0) {$B(1,0)$};
  \node[below] at (0,0) {$(0,0)$};
  \node[above] at (0,{sqrt(3)}) {$(0,\sqrt{3})$};

  % Label the filled (down) region with \(\Omega_{\text{down}}\)
  \node at (0.5,0.1) {$\Omega_{\text{down}}$};
  
  % Label the unfilled right-side (upper) region with \(\Omega\)
  \node at (0.3,1) {$\Omega$};
\end{tikzpicture}
    \caption{Moduli space of triangles}
    \label{fig:moduli-space-of-triangle}
\end{figure}
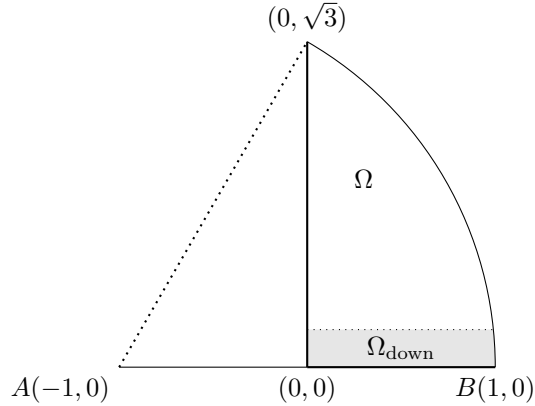

\medskip\medskip

Let $t_0=\tan(\pi/60)$. For the triangle \( T(s,t) \), let \( \mu_k^t(s) \), \( \hat{\mu}_k^t(s) \), and \( \bar{\mu}_k(s) \) denote the \( k \)-th eigenvalues of Problems 1, 2, and 3, respectively, in the same manner as in Section~\ref{sec:preliminary}.
From Theorem \ref{thm:eigenvalue_estimate-cor}, we have the following estimates for \( \mu_k^t(s) \) for each \( s \):

\begin{equation}\label{eq:estimation-for-mu-t-computation}
\underline{\mu}_k(s) := \frac{\bar{\mu}_k(s)}{1 + t_0^{2/3}/(3\pi^2) \cdot \bar{\mu}_k(s)} \leq \mu_k^t(s) \leq \hat{\mu}_k^{t_0}(s) \quad \text{for } t \in (0, t_0].
\end{equation}
The relations among the involved eigenvalues are summarized in Table \ref{tab:relations-eigenvalues}.

\renewcommand{\arraystretch}{1.5}
\begin{table}[h]
    \centering
    \caption{Relations among eigenvalues $\lambda_{k}(s,t)$, $\mu_k^t(s)$, $\hat{\mu}_k^t(s)$ and $\bar{\mu}_k(s)$
    \label{tab:relations-eigenvalues}}    \begin{tabular}{|c|c|c|}
        \hline
        & Eigenvalue & Relations\\
        \hline
   Objective     & $\lambda_k(s,t)$ & 
$\lambda_k(s,t) = t^{-4/3} (\mu_k^t(s)+\pi^2/t^2)$\\
 \hline
        Problem 1 & $\mu_k^t(s)$ & $ \underline{\mu}_k(s)\leq\mu_k^t(s)
%=t^\frac{4}{3}(\lambda_k(s,t)-\pi^2/t^2)$
\leq
\hat{\mu}_k^t(s) $
\\ \hline
        Problem 2 & $\hat{\mu}_k^t(s)$ &  $\hat{\mu}_k^t(s)\leq \hat{\mu}_k^{t_0}(s)$ for $t\leq t_0$\\ \hline
        Problem 3 & $\bar{\mu}_k(s)$ & $\hat{\mu}_k^t(s) \searrow \bar{\mu}_k(s) \quad \text{as } t\to 0+$ \\ \hline
    \end{tabular}

\end{table}

Note that the eigenvalues \(\bar{\mu}_k(s)\) and \(\kappa_k(s)\) are related by 
\begin{equation}
\label{eq:kappa-mu-bar-relation-comput}
        (2\pi^2)^{2/3}\kappa_k(s) =  \bar{\mu}_k(s).
\end{equation}
Regarding the quantity \(\kappa_k(s)\), Ourmi\`eres-Bonafos established the following result:
\begin{lemma}{[\cite{ourmieres2015dirichlet}, Proposition 2.3, 2.5]}
\label{thm:airy-eigen}
The $k$-th eigenvalue $\kappa_k(s)$ is the $k$-th positive solution to the following implicit equation in :
\begin{align}
\label{eq:airy-eigen}
\begin{split}
(f_s(\kappa):=)~&\sqrt[3]{1+s} \mathcal{A}\left((1+s)^{2/3} \kappa\right) \mathcal{A}'\left((1-s)^{2/3} \kappa\right)\\
&- \sqrt[3]{1-s} \mathcal{A}\left((1-s)^{2/3} \kappa\right) \mathcal{A}'\left((1+s)^{2/3} \kappa\right) = 0,
\end{split}
\end{align}
  where $\mathcal{A}$ denotes the Airy reversed function defined by $\mathcal{A}(u)=\text{Ai}(-u)$.

In this paper, we will obtain explicit bounds for $\kappa_3(s)$ by using verified computation methods; see Section \ref{sec:estimation-kappa}.
\end{lemma}

The statement regarding the indexing of the eigenvalues is clear, but it was not explicitly stated in \cite{ourmieres2015dirichlet}. Therefore, a detailed discussion on the equivalence between the eigenvalues of $l_s^\text{mod}$ and the solutions to \eqref{eq:airy-eigen} is given in Lemma \ref{lem:eigen-implicit}.

\medskip\medskip

\textbf{Steps of computer-assisted proof}

We will provide an upper bound of $\hat{\mu}_2^{t_0}(s)$ and a range of $\bar{\mu}_3(s)$ to show that
    \begin{equation}\label{eq:estimation-for-mu-t-computation-steps}
        \mu_2^t(s) \leq \hat{\mu}_2^{t_0}(s) < \frac{\bar{\mu}_3(s)}{1 + t_0^{2/3}/(3\pi^2) \cdot \bar{\mu}_3(s)} \leq \mu_3^t(s) \quad \text{for all } (s,t) \in \Omega_{\text{down}},
    \end{equation}
where $t_0=\tan(\pi/60)$.

Consequently, we ensure that 
    \begin{equation}\label{eq:airy-eigen-steps}
        \mu_2^t(s) <  \mu_3^t(s) \quad \text{for all } (s,t) \in \Omega_{\text{down}},
    \end{equation}
    i.e.,
    \begin{equation}\label{eq:eigen-steps}
        \lambda_2(s,t)  <  \lambda_3(s,t) \quad \text{for all } (s,t) \in \Omega_{\text{down}},
    \end{equation}

The proof proceeds according to the outline below. 
    \begin{description}
        \item[Step 1:] Obtain an upper bound for \( \hat{\mu}_2^{t_0}(s) \) using the Rayleigh--Ritz method.
        
        \item[Step 2:]  Determine the range of $\bar{\mu}_3(s)$  by utilizing the relation \(\bar{\mu}_3(s) = (2\pi^2)^{2/3} \kappa_3(s)\) and the value of \( \kappa_3(s) \) in \eqref{eq:airy-eigen}.
    \end{description}
Each step is performed for each \( s \in [0,1] \). Note that in practical computation, $s$ is taken as small intervals that subdivide $[0,1]$.

The behaviors of $\hat{\mu}_2^{t_0}(s)$, $\bar{\mu}_3(s)$ and $\underline{\mu}_3(s)$ are shown in Figure \ref{fig:mu-hat2-mu-bar3}.
\begin{figure}[H]
    \centering
\includegraphics[width=0.8\linewidth]{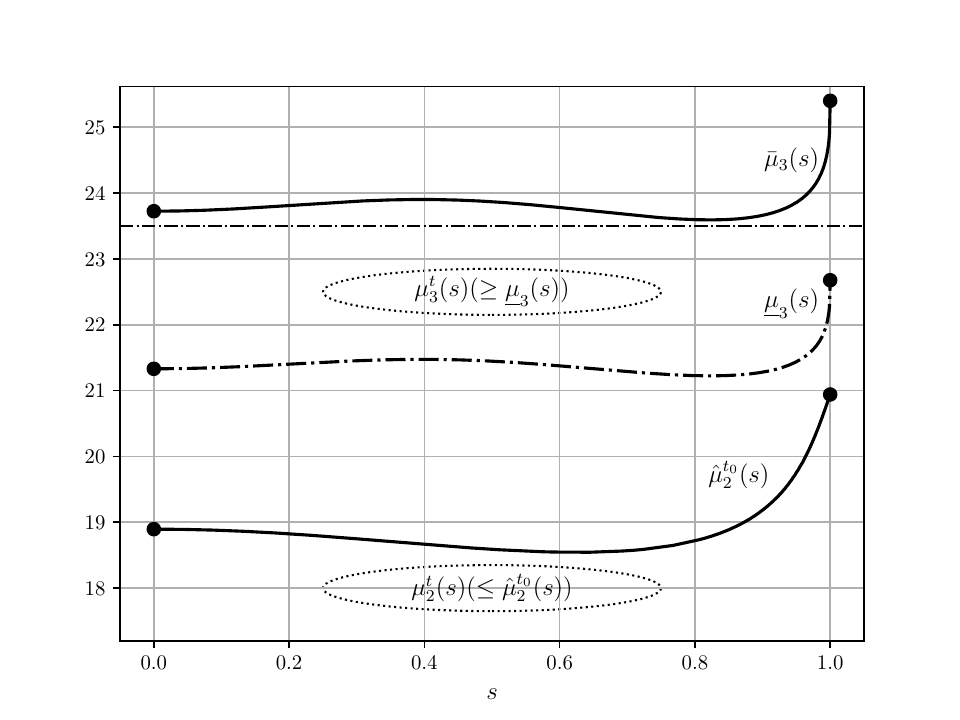}\vspace{-0.5cm}
\begin{picture}(0,0)
\put(-37,205){$\mu = 23.5$}
\end{picture}
    \caption{Graph of $\hat{\mu}_2^{t_0}(s)$, $\bar{\mu}_3(s)$ and $\underline{\mu}_3(s)$ }
    \label{fig:mu-hat2-mu-bar3}
\end{figure}
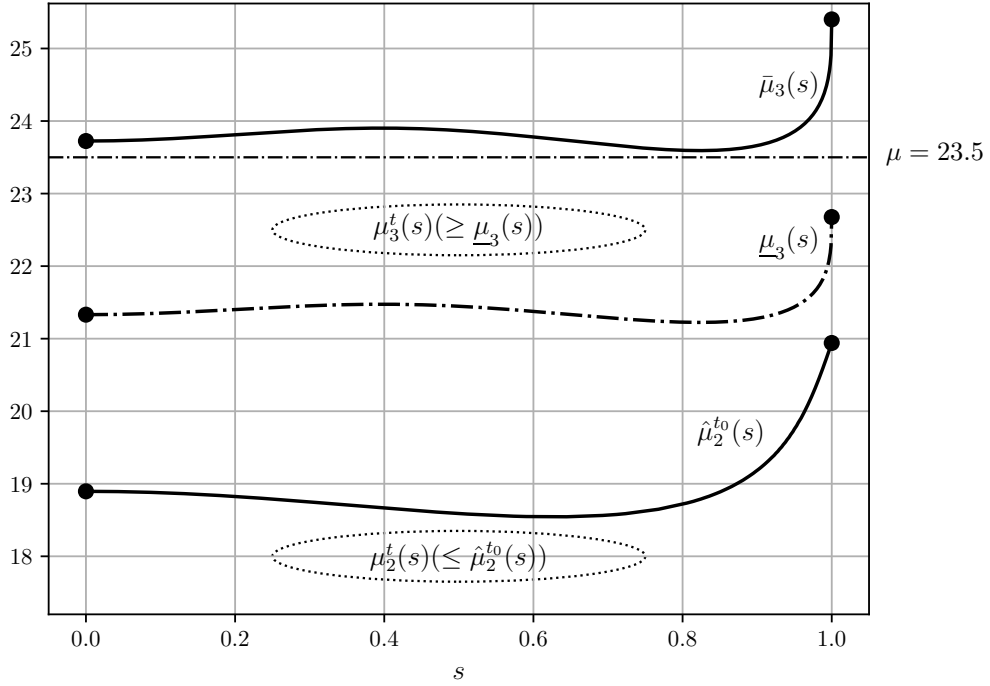

\medskip

\medskip

\subsection{Step 1: Upper bound for $\hat\mu_2^{t_0}(s)$}

The Rayleigh--Ritz method is utilized to obtain the upper bound for the eigenvalue $\hat\mu_2^{t_0}(s)$ over the one-dimensional interval $I^{t_0}$:
$$
I^{t_0}:=\left[a_{\text{left}},a_{\text{right}}\right]\left(= \left[{-t_0^{-2/3} (1+s)},{t_0^{-2/3} (1-s)}\right]\right).
$$

For $i=1,\cdots,N$, define the function $\varphi_i$ by
\begin{equation}
    \varphi_i(x)=(x-a_{\text{left}})(x-a_{\text{right}})\exp(-(x-c_i)^2)~(\in H^1_0(I^{t_0})),
\end{equation}
where
\begin{equation}\label{eq:def-of-phi-i}
    c_i = \frac{(N-i)a_{\text{left}} + (i+N-2)a_{\text{right}}}{2(N-1)}.
\end{equation}

The trial subspace of the Rayleigh--Ritz method is defined as follows:
\begin{equation}
\label{def:fem-space-cg-1d}
\widehat V(I^{t_0}) := \mbox{span}\left\{  \varphi_1,\varphi_2,\cdots,\varphi_N  \right\}(\subset H^1_0(I^{t_0})).
\end{equation}

To estimate the upper bound for the eigenvalue $\hat\mu_2^{t_0}(s)$ of \eqref{eq:problem2}, we consider the following eigenvalues defined over the trial space $\widehat{V}(I^{t_0})$:
\begin{equation}
    \hat{\mu}_{k,h}^{t_0}(s)=\min_{\widehat{V}^{(k)}\subset \widehat{V}(I^{t_0})}\max_{\hat{v}\in \widehat V^{(k)}}R_{t_0}[\hat{u}]~~(k=1,2,\cdots).
\end{equation}
Here, $\widehat{V}^{(k)}$ is $k$-dimensional subspace of $\widehat{V}(I^{t_0})$, and $R_{t_0}$ is the Rayleigh quotient defined in \eqref{eq:def-rt}.

Note that $\hat{\mu}_{2,h}^{t_0}(s)$ serves as an upper bound for $\hat{\mu}_{2,h}^{t_0}(s)$, that is,
\begin{equation*}
\hat{\mu}_2^{t_0}(s) \leq \hat{\mu}_{2,h}^{t_0}(s)~~\forall s\in[0,1].
\end{equation*}
To obtain a uniform upper bound of
\(\hat{\mu}_{2,h}^{t_0}(s)\) with respect to $s \in [0,1]$,
we partition \([0,1]\) into subintervals 
\[
  J_i 
  \;=\; \Bigl[\tfrac{i-1}{N_s}, \tfrac{i}{N_s}\Bigr], 
  \quad 
  i=1,2,\dots,N_s.
\]
Using rigorous interval arithmetic library INTLAB, one can compute an enclosure for \(\hat{\mu}_{2,h}^{t_0}(s)\) on each subinterval \(J_i\).  
The details of computation processes are summarized in Algorithm~\ref{algorithm-upper-bound}.

\begin{algorithm}[H]
\caption{\label{algorithm-upper-bound} Estimation of uniform upper bound of \(\hat{\mu}_2^{t_0}(s)\) for \(s \in [0,1]\).}
\KwIn{A discretization dimension \(N\); number of subintervals \(N_s\) for the variable \(s\in[0,1]\).}
\KwOut{A certified bound \(U\) such that \(\hat{\mu}_2^{t_0}(s)\le U\) for all \(s \in [0,1]\).}
Partition \([0,1]\) into \(N_s\) subintervals \(J_1,\dots,J_{N_s}\)\;

\For{\(i=1\) \textbf{to} \(N_s\)}{
    Construct stiffness and mass matrices of the Rayleigh--Ritz method as interval matrices over \(s \in J_i\)\\
    Solve the eigenvalue problem of the interval matrices to enclose \(\hat{\mu}_{2,h}^{t_0}(s)\) for \(s\in J_i\)\\
    Record the maximum upper bound \(U_i\) for \(\hat{\mu}_{2,h}^{t_0}(s)\) on \(J_i\)
}
\(U\gets \max\{U_1,\dots,U_{N_s}\}\)\;

\Return \(U\)\;
\end{algorithm}

In practical computations, we set the values of \( N \) and \( N_s \) as follows:
\begin{equation}
    N = 17, \quad N_s = 100.
\end{equation}

By combining the results from each subinterval \(J_i\), we obtain the following uniform bound with respect to $s$:
\[
  \hat{\mu}_2^{t_0}(s)\;\le\;\hat{\mu}_{2,h}^{t_0}(s)\;\le\;21.091
  \quad
  \forall\,s\in[0,1].
\]
Therefore, from the estimate \eqref{eq:estimation-for-mu-t-computation}, one can conclude that 
\begin{equation}
\label{eq:upper-bound-for-mu2}
  \mu_2^{t}(s) 
  \;\le\; 
  21.091
  \quad
  \forall\,t \in (0, t_0], \; s \in [0, 1].
\end{equation}

\medskip

\subsection{Step 2: Estimation for \(\overline{\mu}_3(s)\)}\label{sec:estimation-kappa}

Here, we aim to establish the following estimate:  
\begin{equation}\label{eq:mu3-bounds}  
    \overline{\mu}_3(s) \ge  23.5 \quad \text{for all } s \in [0,1).  
\end{equation}  
To verify this inequality, we first confirm that \(\overline{\mu}_3(0) \geq 23.5\) and then ensure that the branch of \(\overline{\mu}_3(s)\) in Figure \ref{fig:mu-hat2-mu-bar3} does not cross the line \(\mu=23.5\). The details of this discussion are as follows:
% This is achieved by demonstrating that the equation \(f_s(\kappa) = 0\) has no solutions in \([22.50, 23.50]\). Through this approach, we conclude that \(\overline{\mu}_3(s) \geq 23.50\) for all \(s \in [0,1)\).

\begin{enumerate}
    \item \textbf{Verification of $\overline{\mu}_3(s)$ at \(s = 0\):} Using the function \texttt{verifynlssall} from the verified computation library INTLAB \cite{Rump1999}, we obtain a range rigorously contains the third positive solution to the equation \eqref{eq:airy-eigen} at $s=0$, that is,
    \[
        \kappa_3(0)\in [3.2481,3.2482].
    \]
    The relation \eqref{eq:kappa-mu-bar-relation} implies that \(\overline{\mu}_3(0)\geq 23.5\). Figure~\ref{fig:f_s_kappa_plot} illustrates the graph of \( f_s(\kappa) \) at \( s = 0 \). 
% Note that for $s\approx 1$ it is difficult to evaluate the range of $\kappa_3(s)$ directly due to its singularity at $s=1$ \cite[\S 2]{ourmieres2015dirichlet}. 
    For the details of the function \texttt{verifynlssall}, see Remark \ref{rem:verifynlssall}.

    \begin{figure}[H]
    \centering
    \includegraphics[width=0.9\linewidth]{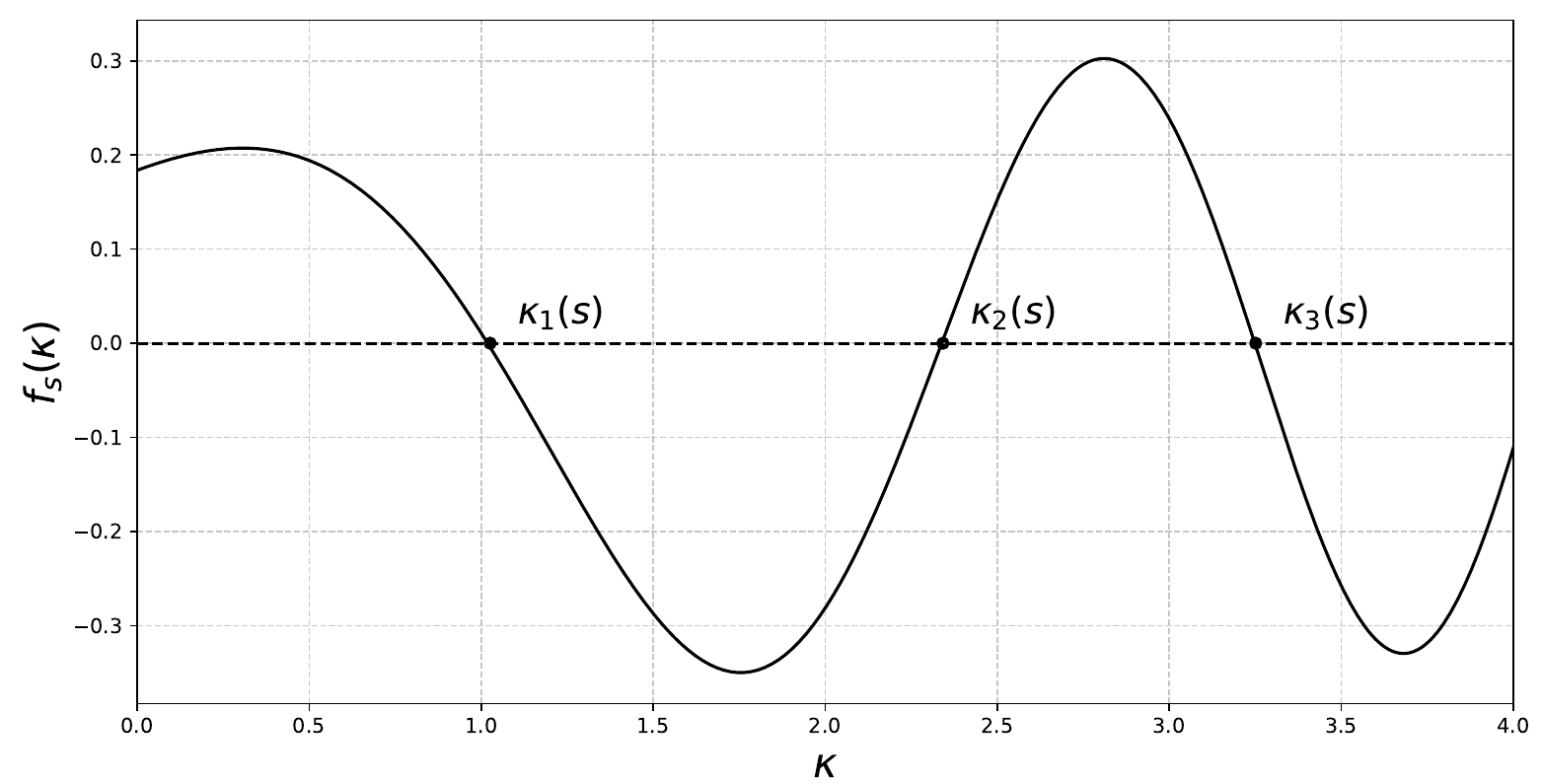}
    \caption{Graph of \( f_s(\kappa) \) at \( s=0 \) for \(\kappa \in [0, 4]\)}
    \label{fig:f_s_kappa_plot}
    \end{figure}
    
    \item \textbf{Investigation of the branch  \(\overline{\mu}_3(s)\) for \(s \in [0,1)\):} By evaluating the range of $f_s(\kappa)$ directly, it is validated that for any $s\in [0,1)$, 
    \[
        f_s(\kappa) > 0 \quad \text{for} \quad \kappa\in [3.2174,3.2175] , 
    \]
\end{enumerate}
which further leads to $\overline{\mu}_3(s)\neq 23.5$ from the relation \eqref{eq:kappa-mu-bar-relation}.
As noted in  Lemma~\ref{eq:bonafos},  
 the mapping \( s \mapsto \kappa_3(s) \) is analytic on \((-1,1)\). Since $\overline{\mu}_3(0)> 23.5$, the above steps ensure that $\overline{\mu}_3(s) > 23.5$
for all \( s \in [0,1) \).

\medskip
\medskip

% \begin{algorithm}[H]
% \caption{\label{algorithm-zero-finding} Verified Zero Finding for $\kappa$ in Equation \eqref{eq:airy-eigen}}
% \begin{algorithmic}[1]
% \For{each subinterval $[s_i^-, s_i^+]$ of $s$, where $i = 1, 2, \ldots, N_s$}
%     \State Partition $[1, 2]$ into $N_\kappa = 400$ smaller subintervals $[\kappa_j^-, \kappa_j^+]$ for $j = 1, 2, \ldots, N_\kappa$\;
%     \For{each subinterval $[\kappa_j^-, \kappa_j^+]$}
%         \State Compute $f_s(\kappa_j^-)$ and $f_s(\kappa_j^+)$\;
%         \If{$f_s(\kappa_j^-) f_s(\kappa_j^+) < 0$}
%             \State Compute $f_s'(\kappa)$ on $[\kappa_j^-, \kappa_j^+]$\;
%             \State Verify that $0 \notin f_s'([\kappa_j^-, \kappa_j^+])$\;
%             \State Record the interval $[\kappa_j^-, \kappa_j^+]$ containing the zero\;
%         \EndIf
%     \EndFor
% \EndFor
% \end{algorithmic}
% \end{algorithm}
Consequently, we obtain the following uniform lower bound for $\mu_3^t(s)$ for all $s \in [0,1)$:
\begin{equation}\label{eq:lower-bound-for-mu3}
    21.149 \leq \inf_{s\in[0,1)} \frac{\overline{\mu}_3(s)}{1 + t_0^{\frac{2}{3}}/(3\pi^2) \cdot \overline{\mu}_3(s)} 
    \leq \mu_3^t(s),
\end{equation}
where $t_0=\tan(\pi/60)$.

From the bounds \eqref{eq:upper-bound-for-mu2} and \eqref{eq:lower-bound-for-mu3}, the eigenvalues $\mu_k^t(s)~(k=2,3)$ satisfy the following inequality:
\begin{equation}
    \mu_2^t(s)\leq 21.091 < 21.149\leq \mu_3^t(s) \quad \text{for any } (s,t) \in \Omega_{\text{down}}.
\end{equation}
From the relation \eqref{eq:relation_mu_and_lambda}, we obtain
\begin{equation}
    \lambda_2(s,t) < \lambda_3(s,t) \quad \text{for any } (s,t) \in \Omega_{\text{down}}.
\end{equation}
Thus, we complete the proof of Theorem \ref{thm:degenerate-case-for-the-main-conjecture}.

\begin{remark}\label{rem:verifynlssall}
The function \texttt{verifynlssall} is a rigorous global solver for nonlinear systems that locates all zeros of a continuous and differentiable function 
\[
f : \mathbb{R}^n \to \mathbb{R}^n
\]
within a prescribed compact box \(x_0 \subset \mathbb{R}^n\). It rigorously demonstrates both the existence and the uniqueness of the solution of the nonlinear equation within the given interval. Specifically, the algorithm produces a collection of disjoint closed boxes, each certified to contain exactly one zero of \(f\). For further theoretical details, refer to \cite{Rump2018,Knueppel1994}.
\end{remark}

\section{Conclusion}

We derived the estimate \eqref{eq:main-estimation} by comparing the Laplacian eigenvalue problem with the eigenvalue problem of the Schr\"{o}dinger operator that emerges as the domain collapses. Using this estimate, we established the simplicity of the second Dirichlet eigenvalue \(\lambda_2\) for triangles with minimum normalized height less than $\tan(\pi/60)/2$. In conjunction with the results from our previous work~\cite{endo2024guaranteed} for triangles with  minimum normalized height greater than $\tan(\pi/60)/2$, we arrive at the following theorem:

\begin{theorem}\label{thm:main-conjecture}
The second Dirichlet eigenvalue is simple for every non-equilateral triangle.
\end{theorem}

\appendix

\section{Properties about eigenvalues of the Schr\"{o}dinger operator}
\label{appendix-airy}

\begin{lemma}\label{lem:problem1-2-appendix}The pair \((\hat{\mu}, \hat{u})\) is an eigen-pair of Problem 2 if and only if \[ \left(\hat{\mu},\, \hat{u}\Bigl(t^{-2/3}(x-s)\Bigr)\sqrt{\frac{2}{h(x)}} \sin\left(\frac{\pi y}{h(x)}\right)\right) \] is an eigen-pair of the following eigenvalue problem:  Find $  u \in W(T^t)$ and  $\hat{\mu} > 0$ such that 
\begin{equation}\label{eq:problem1-in-lemma}
  t^{4/3} \left[ (\nabla u, \nabla v)_{T^t} - \frac{\pi^2}{t^2} (u, v)_{T^t} \right] = \hat{\mu} (u, v)_{T^t} \quad \text{for all } v \in W(T^t).  
\end{equation}
\end{lemma}

\begin{proof}
Suppose that \( (\hat{\mu}, \hat{u}) \) is an eigenpair of Problem 2. Define \( u \in W(T^t) \) by
\begin{equation} \label{eq:u-def}
u(x, y) = \hat{u}(\xi) \, \phi(x, y),
\end{equation}
where
\[
\xi = t^{-2/3}(x - s), \quad \phi(x, y) = \sqrt{\dfrac{2}{h(x)}} \sin\left( \dfrac{\pi y}{h(x)} \right).
\]
Similarly, for any \( \hat{v} \in H^1_0(I^t) \), define \( v(x, y) = \hat{v}(\xi) \, \phi(x, y) \).
First, note that \( u \in H^1_0(T^t) \) because \( \hat{u} \in H^1_0(I^t) \) and \( \phi(x, y) \) vanishes at \( y = 0 \) and \( y = h(x) \).

Our goal is to show that \( u \) satisfies  \eqref{eq:problem1-in-lemma}. We compute the terms \( (\nabla u, \nabla v)_{T^t} \) and \( (u, v)_{T^t} \).
Using the chain rule, the partial derivatives of \( u \) are expressed by
\begin{align*}
\frac{\partial u}{\partial x} &= t^{-2/3} \hat{u}'(\xi) \, \phi(x, y) + \hat{u}(\xi) \, \frac{\partial \phi}{\partial x}, \\
\frac{\partial u}{\partial y} &= \hat{u}(\xi) \, \frac{\partial \phi}{\partial y}.
\end{align*}
Similarly for \( v \).
After integrating over \( y \) and simplifying, we obtain
\begin{align}\label{eq:integral-nabla-l2}
\begin{split}
(\nabla u, \nabla v)_{T^t} 
&= \int_{T^t} \left( \frac{\partial u}{\partial x} \frac{\partial v}{\partial x} + \frac{\partial u}{\partial y} \frac{\partial v}{\partial y} \right) dx \, dy\\
&= \int_{-1}^{1} \left[ t^{-4/3} \hat{u}'(\xi) \hat{v}'(\xi) + \hat{u}(\xi) \hat{v}(\xi) \, W(x) \right] dx, \\
(u, v)_{T^t} &= \int_{-1}^{1} \hat{u}(\xi) \hat{v}(\xi) \, dx,
\end{split}
\end{align}
where
\begin{equation}
    W(x) = \dfrac{\pi^2}{h(x)^2} + \dfrac{(3 + 4\pi^2) [h'(x)]^2}{12 h(x)^2}.
\end{equation}
Next, we change variables from \( x \) to \( \xi \):
\begin{equation}\label{eq:change-variables}
    x = t^{2/3} \xi + s, \quad dx = t^{2/3} d\xi.
\end{equation}
Substituting \eqref{eq:change-variables} into \eqref{eq:integral-nabla-l2}, we get
\begin{align*}
(\nabla u, \nabla v)_{T^t} &= t^{2/3} \int_{I^t} \left[ t^{-4/3} \hat{u}'(\xi) \hat{v}'(\xi) + \hat{u}(\xi) \hat{v}(\xi) \, W(t^{2/3} \xi + s) \right] t^{2/3} d\xi \\
&= \int_{I^t} \left[ \hat{u}'(\xi) \hat{v}'(\xi) + t^{4/3} \hat{u}(\xi) \hat{v}(\xi) \, W(t^{2/3} \xi + s) \right] d\xi, \\
(u, v)_{T^t} &= t^{2/3} \int_{I^t} \hat{u}(\xi) \hat{v}(\xi) \, t^{2/3} d\xi = t^{4/3} \int_{I^t} \hat{u}(\xi) \hat{v}(\xi) \, d\xi.
\end{align*}
Thus, it follows that
\begin{align}
&t^{4/3} \left[ (\nabla u, \nabla v)_{T^t} - \dfrac{\pi^2}{t^2} (u, v)_{T^t} \right]\\
&~~~~~=
\int_{I^t} \left[ \hat{u}'(\xi) \hat{v}'(\xi) + \left( t^{4/3} W(t^{2/3} \xi + s) - \dfrac{\pi^2}{t^2} t^{4/3} \right) \hat{u}(\xi) \hat{v}(\xi) \right] d\xi.
\end{align}
Simplifying the potential term, we have
\begin{align}
&t^{4/3} \left( W(t^{2/3} \xi + s) - \dfrac{\pi^2}{t^2} \right) \\
&~~= t^{4/3} \left( \dfrac{\pi^2}{h(t^{2/3} \xi + s)^2} + \dfrac{(3 + 4\pi^2) [h'(t^{2/3} \xi + s)]^2}{12 h(t^{2/3} \xi + s)^2} - \dfrac{\pi^2}{t^2} \right)(=:V(t, \xi)).
\end{align}
Therefore, the left-hand side of \eqref{eq:problem1-in-lemma} becomes
\[
\int_{I^t} \left[ \hat{u}'(\xi) \hat{v}'(\xi) + V(t, \xi) \hat{u}(\xi) \hat{v}(\xi) \right] d\xi.
\]
The right-hand side of \eqref{eq:problem1-in-lemma} simplifies to
\[
\hat{\mu} \,  \int_{I^t} \hat{u}(\xi) \hat{v}(\xi) \, d\xi.
\]
Since the factor \( t^{4/3} \) appears on both sides, it can be canceled out, leading to the variational formulation
\[
\int_{I^t} \left[ \hat{u}'(\xi) \hat{v}'(\xi) + V(t, \xi) \hat{u}(\xi) \hat{v}(\xi) \right] d\xi = \hat{\mu} \int_{I^t} \hat{u}(\xi) \hat{v}(\xi) \, d\xi,
\]
which is exactly the eigenvalue problem of Problem 2.

Conversely, if \( (\hat{\mu}, u) \) is an eigenpair of \eqref{eq:problem1-in-lemma} with \( u \in W(T^t) \), then \( u \) can be expressed in the form \eqref{eq:u-def}. Reversing the above steps shows that \( (\hat{\mu}, \hat{u}) \) satisfies Problem 2. Thus, the eigenpairs of Problem 2 correspond precisely to those of \eqref{eq:problem1-in-lemma} within \( W(T^t) \).
\end{proof}

% \section*{Appendix B}

Let us consider the operator
\begin{equation}\label{eq:def-of-l-s-mod}
    l_s^{\text{mod}} := -\frac{d^2}{dx^2} + v_s^{\mathrm{mod}}(x), ~\text{with} ~ v_s^{\mathrm{mod}}(x) := \left( \frac{1}{s+1}\,\mathbf{1}_{\mathbb{R}_-}(x) + \frac{1}{s-1}\,\mathbf{1}_{\mathbb{R}_+}(x) \right) |x|,
\end{equation}
defined on \(H^2(\mathbb{R})\).

In this section, we prove the following lemma showing that a positive number \(\kappa\) is an eigenvalue of \(l_s^{\mathrm{mod}}\) if and only if it satisfies an implicit equation involving the Airy function.

\begin{lemma}\label{lem:eigen-implicit}
For fixed \(s\in(-1,1)\), a positive real number \(\kappa\) is an eigenvalue of \(l_s^{\text{mod}}\) if and only if it satisfies
\begin{equation}\label{eq:implicit}
(1+s)^{1/3}\,\mathcal{A}\Bigl((1+s)^{2/3}\kappa\Bigr)\,\mathcal{A}'\Bigl((1-s)^{2/3}\kappa\Bigr)
+(1-s)^{1/3}\,\mathcal{A}\Bigl((1-s)^{2/3}\kappa\Bigr)\,\mathcal{A}'\Bigl((1+s)^{2/3}\kappa\Bigr)
=0,
\end{equation}
where
\[
\mathcal{A}(x)=\operatorname{Ai}(-x),\qquad \mathcal{A}'(x)=-\operatorname{Ai}'(-x),
\]
and \(\operatorname{Ai}\) is the standard Airy function.
\end{lemma}

\begin{proof}
Assume first that \(\kappa>0\) is an eigenvalue of \(l_s^{\mathrm{mod}}\) with eigenfunction \(\Psi\in H^2(\mathbb{R})\) satisfying
\[
l_s^{\mathrm{mod}}\Psi=\kappa\Psi.
\]
Since the potential \(\hat{v}_s^{\mathrm{mod}}(x)\) is defined piecewise, we decompose \(\Psi\) as
\[
\Psi(x)=
\begin{cases}
\Psi^-(x), & x<0,\\[1mm]
\Psi^+(x), & x>0.
\end{cases}
\]

\medskip
For \(x<0\),  the eigenvalue equation reads
\[
-\Psi''(x)-\frac{1}{s+1}\,x\,\Psi(x)=\kappa\Psi(x).
\]
Introduce the change of variables
\[
\xi=(1+s)^{-1/3}\Bigl(x+\kappa(1+s)\Bigr).
\]
A short computation shows that this transforms the equation into the standard Airy equation:
\begin{equation}
\Psi''(\xi)+\xi\,\Psi(\xi)=0  .
\end{equation}
Hence, the square-integrable solution on \((-\infty,0)\) is
\begin{equation}\label{eq:Psi-minus}
\Psi^-(x)=\alpha^-\,\mathcal{A}\Bigl((1+s)^{-1/3}\bigl(x+\kappa(1+s)\bigr)\Bigr),
\end{equation}
for some constant \(\alpha^-\in\mathbb{R}\). In particular, at \(x=0\) we have
\[
\Psi^-(0)=\alpha^-\,\mathcal{A}\Bigl((1+s)^{2/3}\kappa\Bigr).
\]

\medskip
For \(x>0\), the eigenvalue equation becomes
\[
-\Psi''(x)-\frac{1}{1-s}\,x\,\Psi(x)=\kappa\Psi(x).
\]
To reduce this to the standard Airy equation, we set
\[
\xi=(1-s)^{-1/3}\Bigl(\kappa(1-s)-x\Bigr).
\]
Then one may check that the square-integrable solution on \((0,\infty)\) is given by
\begin{equation}\label{eq:Psi-plus}
\Psi^+(x)=\alpha^+\,\mathcal{A}\Bigl((1-s)^{-1/3}\bigl(\kappa(1-s)-x\bigr)\Bigr),
\end{equation}
with \(\alpha^+\in\mathbb{R}\). In particular, at \(x=0\) we obtain
\[
\Psi^+(0)=\alpha^+\,\mathcal{A}\Bigl((1-s)^{2/3}\kappa\Bigr).
\]

\medskip
Since \(\Psi\in H^2(\mathbb{R})\), both \(\Psi\) and its derivative must be continuous at \(x=0\). Hence, the matching conditions
\[
\alpha^-\,\mathcal{A}\Bigl((1+s)^{2/3}\kappa\Bigr)
-\alpha^+\,\mathcal{A}\Bigl((1-s)^{2/3}\kappa\Bigr)
=0,
\]
and
\[
\alpha^-\,(1+s)^{-1/3}\,\mathcal{A}'\Bigl((1+s)^{2/3}\kappa\Bigr)
+\alpha^+\,(1-s)^{-1/3}\,\mathcal{A}'\Bigl((1-s)^{2/3}\kappa\Bigr)
=0,
\]
must hold. In matrix form, this system reads
\[
\begin{pmatrix}
\mathcal{A}((1+s)^{2/3}\kappa) & -\mathcal{A}((1-s)^{2/3}\kappa)\\[1mm]
(1+s)^{-1/3}\,\mathcal{A}'((1+s)^{2/3}\kappa) & (1-s)^{-1/3}\,\mathcal{A}'((1-s)^{2/3}\kappa)
\end{pmatrix}
\begin{pmatrix}\alpha^-\\[1mm]\alpha^+\end{pmatrix}
=\begin{pmatrix}0\\0\end{pmatrix}.
\]
A basic result in linear algebra shows that a nontrivial solution exists if and only if the determinant of the coefficient matrix vanishes. A short calculation reveals that this determinant, up to the positive factor \((1+s)^{1/3}(1-s)^{1/3}\), is exactly
\[
(1+s)^{1/3}\,\mathcal{A}((1+s)^{2/3}\kappa)\,\mathcal{A}'((1-s)^{2/3}\kappa)
+(1-s)^{1/3}\,\mathcal{A}((1-s)^{2/3}\kappa)\,\mathcal{A}'((1+s)^{2/3}\kappa).
\]
Thus, a nontrivial solution exists if and only if the implicit equation \eqref{eq:implicit} holds.

\medskip
Conversely, if \(\kappa>0\) satisfies \eqref{eq:implicit}, then the coefficient matrix is singular so that a nontrivial pair \((\alpha^-,\alpha^+)\neq(0,0)\) exists. Defining
\[
\Psi(x)=
\begin{cases}
\alpha^-\,\mathcal{A}\Bigl((1+s)^{-1/3}\bigl(x+\kappa(1+s)\bigr)\Bigr), & x<0,\\[1mm]
\alpha^+\,\mathcal{A}\Bigl((1-s)^{-1/3}\bigl(\kappa(1-s)-x\bigr)\Bigr), & x>0,
\end{cases}
\]
one verifies that \(\Psi\in H^2(\mathbb{R})\) is continuously differentiable at \(x=0\) and satisfies
\[
l_s^{\text{mod}}\Psi=\kappa\Psi.
\]
Hence, \(\kappa\) is an eigenvalue of \(l_s^{\mathrm{mod}}\).

\end{proof}

\section*{Acknowledgement}
Both authors are supported by Japan Society for the Promotion of Science. The first author is supported by JSPS KAKENHI Grant Number JP24KJ1170. The last author is supported by JSPS KAKENHI Grant Numbers JP22H00512, JP24K00538, JP21H00998 and JPJSBP120237407.

\bibliographystyle{elsarticle-num} 
\bibliography{references}

\begin{thebibliography}{10}
\expandafter\ifx\csname url\endcsname\relax
  \def\url#1{\texttt{#1}}\fi
\expandafter\ifx\csname urlprefix\endcsname\relax\def\urlprefix{URL }\fi
\expandafter\ifx\csname href\endcsname\relax
  \def\href#1#2{#2} \def\path#1{#1}\fi

\bibitem{henrot2017shape}
A.~Henrot, Shape optimization and spectral theory, De Gruyter Open Poland, 2017.

\bibitem{endo2024guaranteed}
R.~Endo, X.~Liu, Rigorous estimation for the difference quotients of multiple eigenvalues, arXiv preprint arXiv:2305.14063 (2024).

\bibitem{jerison1991first}
D.~Jerison, The first nodal line of a convex planar domain, International Mathematics Research Notices 1991~(1) (1991) 1--5.

\bibitem{jerison1995diameter}
D.~Jerison, The diameter of the first nodal line of a convex domain, Annals of Mathematics 141~(1) (1995) 1--33.

\bibitem{friedlander2009spectrum}
L.~Friedlander, M.~Solomyak, On the spectrum of the dirichlet laplacian in a narrow strip, Israel journal of mathematics 170 (2009) 337--354.

\bibitem{ourmieres2015dirichlet}
T.~Ourmi{\`e}res-Bonafos, Dirichlet eigenvalues of asymptotically flat triangles, Asymptotic Analysis 92~(3-4) (2015) 279--312.

\bibitem{liu2024guaranteed}
X.~Liu, Guaranteed Computational Methods for Self-Adjoint Differential Eigenvalue Problems, Springer Nature, 2024.

\bibitem{liu2015framework}
X.~Liu, A framework of verified eigenvalue bounds for self-adjoint differential operators, Appl. Math. Comput. 267 (2015) 341--355.

\bibitem{Rump1999}
S.~M. Rump, Intlab—interval laboratory, in: Developments in Reliable Computing, Springer, 1999, pp. 77--104.

\bibitem{Rump2018}
S.~M. Rump, Mathematically rigorous global optimization in floating-point arithmetic, Optimization Methods \& Software 33~(4-6) (2018) 771--798.

\bibitem{Knueppel1994}
O.~Knueppel, Einschließungsmethoden zur bestimmung der nullstellen nichtlinearer gleichungssysteme und ihre implementierung, Ph.D. thesis, Hamburg University of Technology, phD thesis (1994).

\end{thebibliography}
						 
\end{document}